\documentclass[12pt,reqno]{amsart}
 \usepackage[utf8]{inputenc}
 \usepackage[T1]{fontenc}
 \usepackage[usenames, dvipsnames]{color}
 \usepackage{ulem}

 \usepackage{dsfont, amsfonts, amsmath, amssymb,amscd, stmaryrd, latexsym, amsthm, dsfont}
 \usepackage[frenchb,english]{babel}
 \usepackage{enumerate}
 \usepackage{longtable}
 \usepackage{geometry}
 \usepackage{float}
 \usepackage{tikz}
 \usetikzlibrary{shapes,arrows}
 \usepackage{pifont}
 \usepackage{float}
 \usepackage{tikz}
 \usepackage{xypic}
 \usetikzlibrary{shapes,arrows}
 \newtheorem{theorem}{Theorem}[section]
 \newtheorem{lemma}[theorem]{Lemma}

 \theoremstyle{remark}

 \usepackage{hyperref}
 \hypersetup{
 	colorlinks=true,
 	urlcolor=blue,
 	citecolor=blue}
 \def\NN{\mathds{N}}
 \def\RR{\mathbb{R}}
 \def\QQ{\mathbb{Q}}
 \def\CC{\mathbb{C}}
 \def\ZZ{\mathbb{Z}}
 \def\kk{\mathds{k}}
 
\begin{document}
 	
 	\def\NN{\mathbb{N}}
 	\def\RR{\mathds{R}}
 	\def\HH{I\!\! H}
 	\def\QQ{\mathbb{Q}}
 	\def\CC{\mathds{C}}
 	\def\ZZ{\mathbb{Z}}
 	\def\DD{\mathds{D}}
 	\def\OO{\mathcal{O}}
 	\def\kk{\mathds{k}}
 	\def\KK{\mathbb{K}}
 	\def\ho{\mathcal{H}_0^{\frac{h(d)}{2}}}
 	\def\LL{\mathbb{L}}
 	\def\L{\mathds{k}_2^{(2)}}
 	\def\M{\mathds{k}_2^{(1)}}
 	\def\k{\mathds{k}^{(*)}}
 	\def\l{\mathds{L}}

 	\selectlanguage{english}

 \title[The unit group....
 ]{The unit group and the 2-class number of some fields of the form $\mathbb{Q}(\sqrt{2},  \sqrt{pq}, \sqrt{ps})$ and $\mathbb{Q}(\sqrt{2},  \sqrt{pq}, \sqrt{ps}, \sqrt{-\ell})$
 	}

 		\author[M. B. T. El Hamam]{Moha Ben Taleb El Hamam}
 	\address{Moha Ben Taleb El Hamam: 	Sidi Mohamed Ben Abdellah University, Faculty of Sciences Dhar El Mahraz, Fez, Morocco.}
 	\email{mohaelhomam@gmail.com }


 	\begin{abstract}
		Let   $\LL^+=\mathbb{Q}(\sqrt{2},  \sqrt{pq}, \sqrt{ps})$ and $\LL=\mathbb{Q}(\sqrt{2},  \sqrt{pq}, \sqrt{ps}, \sqrt{-\ell})$ be two  fields, where  $q$,  $p$ and $s$ three different prime integers and  $\ell\geq1$ be a positive odd  square-free integer relatively prime to $q$, $p$ and $s$.
	The purpose of this paper is to   show  how one can proceed to perform
	the  calculation  of unit group of the fields of the form $\LL^+=\mathbb{Q}(\sqrt{2},  \sqrt{pq}, \sqrt{ps})$ and $\LL=\mathbb{Q}(\sqrt{2},  \sqrt{pq}, \sqrt{ps}, \sqrt{-\ell})$. More
	precisely, we compute the unit group and the $2$-class number of these fields whenever 	$p\equiv-s\equiv 5\pmod 8, q\equiv7\pmod 8 ~~ \text{and} ~~ \left(\frac{	p}{	q}\right)=\left(\frac{	p}{	s}\right)=\left(\frac{s}{	q}\right)=1$ and
	$\left(\frac{	p}{	q}\right)=\left(\frac{	p}{	s}\right),$ or $p\equiv-s\equiv 5\pmod 8, q\equiv7\pmod 8 ~~ \text{and} ~~ \left(\frac{	p}{	q}\right)=\left(\frac{	p}{	s}\right)=\left(\frac{s}{	q}\right)=-1$.
 	\end{abstract}
 	
 	\selectlanguage{english}
 	
 	\maketitle

  \section{ Introduction}
  
  One major  problem in algebraic number theory (and thus in the theory of units of number fields which is related to all areas of algebraic number theory) is the computation of a
  fundamental system of units.
  Let  $k$ be a  number field of degree $n$ and $k^+$ its maximal real subfield. Let  $E_k$  (resp. $E_{k^+}$) denote  the unit group of $k$ (resp. $k^+$) that is  the group of the invertible elements of the ring $\mathcal{O}_k$  (resp. $\mathcal{O}_{k^+}$) of algebraic integers of the number field  $k$ (resp. $k^+$). The computation of the unit group of $k$ may
  give answers on many questions of algebraic number theory such as the size of the
  $2$-class group, the length of the Hilbert $2$-class field tower of its subextensions and
  the capitulation problem.
  For  quadratic fields, the problem
  is easily solved.    For quartic bicyclic fields, Kubota \cite{Ku-56} gave
  a method for finding a fundamental system of units.
  Wada \cite{Wa-66} generalized Kubota's method, creating an algorithm for computing
  fundamental units in any given multiquadratic field. However, in general, it is not easy to  compute the unit group of a number field  especially for number fields of degree more than $4$. In some of their recent works, Azizi, Chems-Eddin,  El Hamam and Zekhnini used some technical computations to determine the unit group of some  number fields $k$ of degree $8$ and $16$ (see \cite{ChemsUnits9,chemszekhniniazizilambdas,chemszekhniniaziziUnits1,CEH,El Hamam,El Hamam1}). 
  We note that all the fields considered in the previous works are all of the form
  $ \mathbb{Q}(\sqrt{2},  \sqrt{p}, \sqrt{q} )$ or $ \mathbb{Q}(\sqrt{2},  \sqrt{p}, \sqrt{q}, \sqrt{-\ell})$, where $p$ and $q$ are two different prime integers and $\ell$  a positive odd  square-free integer.
  
  Let  $q$,  $p$ and $s$ be  three different primes integers and   $\ell\geq1$ be a positive odd  square-free integer relatively prime to $q$, $p$ and $s$. In the present work,  we shall    determine  the unit group  and the $2$-class number of some fields of degree $8$ and $16$ of the form  $\LL^+= \mathbb{Q}(\sqrt{2},  \sqrt{pq}, \sqrt{ps})$ and
  $ \LL=\mathbb{Q}(\sqrt{2},  \sqrt{pq}, \sqrt{ps}, \sqrt{-\ell})$, where $q$,  $p$ and $s$ satisfy one of the following conditions:
   
  \begin{equation}\label{case1}
  	p\equiv-s\equiv 5\pmod 8, q\equiv7\pmod 8 ~~ \text{and} ~~ \left(\frac{	p}{	q}\right)=\left(\frac{	p}{	s}\right)=\left(\frac{s}{	q}\right)=1,
  \end{equation}
  \begin{equation}\label{case2}
  	 p\equiv-s\equiv 5\pmod 8, q\equiv7\pmod 8 ~~ \text{and} ~~ \left(\frac{	p}{	q}\right)=\left(\frac{	p}{	s}\right)=\left(\frac{s}{	q}\right)=-1.
  \end{equation}

  We note that the computation of the unit group  of these fields may be very important to deal with the problem of the $2$-class field tower of biquadratic number fields (see for example \cite{{acz}}). Furthermore,  these fields are useful in the study of  cyclotomic $\ZZ_2$-extension of the fields $ \mathbb{Q}( \sqrt{pq}, \sqrt{ps})$ and
  $ \mathbb{Q}( \sqrt{pq}, \sqrt{ps}, \sqrt{-\ell})$ and this maybe be useful (see \cite{ChemsUnits9}).

  In the rest of this article we use the following notations: let $\varepsilon_m$ denote the fundamental unit of the quadratic field $\QQ(\sqrt{m})$, $\left(\frac{\cdot}{	\cdot}\right)$ the 
  Legendre Symbol and $h(k)$ (resp. $h_2(k)$) the class  number (resp. $2$-class number) of a number field $k$. 	
  Finally, let
  $h_2(d)$ denote the 2-class number of the quadratic field $\QQ(\sqrt{d})$
  and  $\zeta_n$ denote a primitive $n$-th root of unity.

  \section{ Preliminary results}
  In this section we recall some results that will be useful in what follows.

  \begin{lemma}\label{Lemme azizi} Let $K_0$ be an abelian field,  $K=K_0(i)$ a quadratic extension of $K_0$,  $n\geq 2$  an integer and $\xi_n$ a primitive  $2^n$-th root of unity,  then
  	$\xi_n=\frac{1}{2}(\mu_n+\lambda_ni)$,  where $\mu_n=\sqrt{2+\mu_{n-1}}$,  $\lambda_n=\sqrt{2-\mu_{n-1}}$,  $\mu_2=0$,  $\lambda_2=2$ and $\mu_3=\lambda_3=\sqrt{2}$. Let $n_0$ be the greatest
  	integer such that $\xi_{n_0}$ is contained in $K$,  $\{\epsilon_1,  \cdots,  \epsilon_r\}$ a fundamental system of units of $K_0$ and $\epsilon$ a unit of $K_0$ such that
  	$(2+\mu_{n_0})\epsilon$ is a square in $K_0\ ($if it exists$)$. Then a fundamental system of units of $K$ is one of the following systems:
  	\begin{enumerate}[\rm 1.]
  		\item $\{\epsilon_1,  \cdots,  \epsilon_{r-1},  \sqrt{\xi_{n_0}\epsilon } \}$ if $\epsilon$ exists,   in this case $\epsilon=\epsilon_1^{j_1}\cdots \epsilon_{r-1}^{j_{r-1}}\epsilon_r$,
  		where $j_i\in \{0,  1\}$.
  		\item $\{\epsilon_1,  \cdots,  \epsilon_r \}$ otherwise.
  		
  	\end{enumerate}
  \end{lemma}
  \begin{proof}
  	See \cite[Proposition 2]{azizunints99}.
  \end{proof}
\begin{lemma}\label{Lemme azizi2}  Let  $K_0/\QQ$ be an abelian extension such that $K_0$ is real and $\beta$  a positive square-free algebraic integer of   $K_0$.
	Assume that  $K=K_0(\sqrt{-\beta})$  is a quadratic extension of     $K_0$, which is abelian over $\QQ$. Assume furthermore that $i=\sqrt{-1}\not\in K$.
	Let $\{\epsilon_1, \cdots, \epsilon_r\}$ be a fundamental system of units of      $K_0$. Without loss of generality we may suppose that the units   $\epsilon_i$ are positive.
	Let $\epsilon$ be a unit of $K_0$ such that
	$\beta\epsilon$ is a square in $K_0\ ($if it exists$)$. Then a fundamental system of units of $K$ is one of the following systems :
	\begin{enumerate}[\rm 1.]
		\item $\{\epsilon_1,\cdots,\epsilon_{r-1},\sqrt{-\epsilon } \}$ if $\epsilon$ exists, in this case $\epsilon=\epsilon_1^{j_1}\cdots \epsilon_{r-1}^{j_{r-1}}\epsilon_r$,
		where $j_i\in \{0,1\}$.
		\item $\{\epsilon_1,\cdots,\epsilon_r \}$ otherwise.
		
	\end{enumerate}
\end{lemma}
\begin{proof}
	See \cite[Proposition 3]{azizunints99}.
\end{proof}

Let us first recall the result stated by Wada \cite{Wa-66} to determine a fundamental system of units of a multiquadratic number field.   Let $K_0$ be a multiquadratic number field. Denote by $\sigma$ and $\tau$ two different nontrivial elements of the group $\mathrm{Gal}(K_0/\QQ)$, let then $K_1$, $K_2$ and $K_3$ be respectively the invariant subfields of $K_0$ by $\sigma$,  $\tau$ and $\sigma\tau$, and $E_{K_i}$ the unit group of $K_i$. Then the unit group $E_{K_0}$ of $K_0$ is generated by the elements of each $E_{K_i}$ and the square roots of elements of the product $E_{K_1}E_{K_2}E_{K_3}$ that are perfect squares in $K_0$.  \label{remainder on units}

\bigskip

   Now, we can start our computations of the unit groups of $\LL^+$ and $\LL$. Consider the
   following diagram of subfields of $\LL^+/\QQ(\sqrt 2)$
   \begin{figure}[H]
   	$$\xymatrix@R=0.8cm@C=0.3cm{
   		&\LL^+=\QQ( \sqrt 2,  \sqrt{pq},  \sqrt{ps})\ar@{<-}[d] \ar@{<-}[dr] \ar@{<-}[ld] \\
   		L_1=\QQ(\sqrt 2, \sqrt{pq})\ar@{<-}[dr]& L_2=\QQ(\sqrt 2,  \sqrt{ps}) \ar@{<-}[d]& L_3=\QQ(\sqrt 2,  \sqrt{qs})\ar@{<-}[ld]\\
   		&\QQ(\sqrt 2)}$$
   	\caption{Subfields of $\LL^+/\QQ(\sqrt 2)$}\label{fig:1}
   \end{figure}

   Let $\sigma_1,$   $\sigma_2$, and  $\sigma_3$ be the elements of $ \mathrm{Gal}(\LL^+/\QQ)$ defined by	
   $$	\begin{array}{ccc}
   	\sigma_1(\sqrt{2})=-\sqrt{2},&  \sigma_1(\sqrt{pq})=\sqrt{pq}, &\sigma_1(\sqrt{ps})=\sqrt{ps},\\
   	\sigma_2(\sqrt{2})=\sqrt{2},&  \sigma_2(\sqrt{pq})=-\sqrt{pq}, &\sigma_2(\sqrt{ps})=\sqrt{ps},\\
   	\sigma_3(\sqrt{2})=\sqrt{2},&  \sigma_3(\sqrt{pq})=\sqrt{pq}, &\sigma_3(\sqrt{ps})=-\sqrt{ps}.
   \end{array}$$
   Note that $ \mathrm{Gal}(\LL^+/\QQ)=\langle \sigma_1,   \sigma_2,   \sigma_3\rangle$ and the subfields $L_1$, $L_2$ and $L_3$ are fixed by
   $\sigma_1$, $\sigma_2$ and $\sigma_2\sigma_3$ respectively.

 \section{\textbf{The case: $p$, $q$ and $s$ satisfy conditions (\ref{case1})}}
 
 Let us now recall the   following class number formula for a multiquadratic number field. 
 
 \begin{lemma}[\cite{Ku-50}]\label{wada's f.}
 	Let $K$ be a multiquadratic number field of degree $2^n$, $n\in\mathbb{N}$,  and $k_i$ the $s=2^n-1$ quadratic subfields of $K$. Then
 	$$h(K)=\frac{1}{2^v}q(K)\prod_{i=1}^{s}h(k_i),$$
 	where  $ q(K)=[E_K: \prod_{i=1}^{s}E_{k_i}]$ and   $$     v=\left\{ \begin{array}{cl}
 		n(2^{n-1}-1); &\text{ if } K \text{ is real, }\\
 		(n-1)(2^{n-2}-1)+2^{n-1}-1 & \text{ if } K \text{ is imaginary.}
 	\end{array}\right.$$
 \end{lemma}

 To use the above lemma we need to recall the values of the $2$-class numbers of some quadratic fields.
 \begin{lemma}\label{class numbers of quadratic field}
 	Let $p\equiv -s \equiv 5\pmod 8$ and $q\equiv 7\pmod 8$ be  three different  primes. Then
 	\begin{enumerate}[\rm $\bullet$]
 	\item By \cite[Corollary 18.4]{connor88}, we have  $ h_2(2)=h_2(qs)=1$.
 	\item By \cite[Corollary 19.7]{connor88}, we have $h_2(pq) =h_2(ps)=h_2(2pq)=h_2(2ps)=2$.
 	\item  By \cite[p. 345]{kaplan76}, we have $h_2(2qs)=2$.
 		\end{enumerate}
 	
 \end{lemma}
 
 Let us recall  with the following lemmas which will be useful later.
  
 \begin{lemma}\cite[Lemma 5]{azizi2000}\label{lm}
 	Let $d>1$  be a square-free integer and $\varepsilon_{d} = x + y\sqrt{d}$, where
 	$x$, $y$ are integers or semi-integers. If $N(\varepsilon_{d}) = 1$, then $2(x + 1)$, $2(x-1)$, $2d(x + 1)$ and
 	$2d(x-1)$ are not squares in $\QQ$.
 \end{lemma}

  \begin{lemma}\label{lm 4}
 	Let	$p\equiv -q\equiv 5\pmod 8$ be two primes such that $\left(\frac{	p}{	q}\right)=1$.
 	\begin{enumerate}[\rm 1.]
 		\item Let  $x$ and $y$   be two integers such that
 		$ \varepsilon_{2pq}=x+y\sqrt{2pq}$. Then
 		\begin{enumerate}[\rm a.]
 			\item $p(x-1)$ is a square in $\NN$,
 			\item 	  $\sqrt{2\varepsilon_{2pq}}=y_1\sqrt{p} +y_2\sqrt{2q}$ and 	$2= -py_1^2+2qy_2^2$,  for some integers $y_1$ and $y_2$.
 		\end{enumerate}
 		\item  There are two integers     $a$ and $b$   such that
 		$ \varepsilon_{pq}=a+b\sqrt{pq}$. Then
 		\begin{enumerate}[\rm a.]
 			\item $2p(a+1)$ is a square in $\NN$,
 			\item  $\sqrt{\varepsilon_{pq}}=b_1\sqrt{p}+b_2\sqrt{q}$ and $1=pb_1^2-qb_2^2$,  for some integers $b_1$ and $b_2$.
 		\end{enumerate}
 	\end{enumerate}
 \end{lemma}	
 \begin{proof}
 	See \cite[Lemma 2.6]{ChemsUnits9}.
 \end{proof}

 \begin{lemma}\label{lm 5}
 	Let	$q_{1}\equiv 7\pmod 8$   and	$q_{2}\equiv  3\pmod 8$	 be two primes such that $\left(\frac{	q_{2}}{	q_{1}}\right)=1$.
 	\begin{enumerate}[\rm 1.]
 		\item Let  $x$ and $y$   be two integers such that
 		$ \varepsilon_{q_{1}q_{2}}=x+y\sqrt{q_{1}q_{2}}$. Then
 		\begin{enumerate}[\rm a.]
 			\item $2q_{1}(x-1)$ is a square in $\NN$,
 			\item 	  $\sqrt{\varepsilon_{q_{1}q_{2}}}=y_1\sqrt{q_{1}} +y_2\sqrt{q_{2}}$ and 	$1= -q_{1}y_1^2+q_{2}y_2^2$,  for some integers $y_1$ and $y_2$ such that $y=2y_1y_2$.
 		\end{enumerate}

 		\item  There are two integers     $a$ and $b$   such that
 		$ \varepsilon_{2q_{1}q_{2}}=a+b\sqrt{2q_{1}q_{2}}$. Then
 		\begin{enumerate}[\rm a.]
 			\item $2q_{1}(a-1)$ is a square in $\NN$,
 			\item  $\sqrt{2 \varepsilon_{2q_{1}q_{2}}}=b_1\sqrt{2q_{1}}+b_2\sqrt{q_{2}}$ and $2=-2q_{1}b_1^2+q_{2}b_2^2$,  for some integers $b_1$ and $b_2$ such that $b=b_1b_2$.
 		\end{enumerate}
 	\end{enumerate}
 \end{lemma}	
 \begin{proof}
 	See \cite[Lemma 7]{chemszekhniniazizilambdas}.
 \end{proof}
We will show the following  results that will be useful in the sequel.
 \begin{lemma}\label{lm 6}
	Let	$p\equiv 5\pmod 8$   and	$q\equiv  7\pmod 8$	 be two primes such that $\left(\frac{	p}{	q}\right)=1$.
	\begin{enumerate}[\rm 1.]
		\item Let $x$ and $y$ be two integers  such that
		$ \varepsilon_{2pq}=x+y\sqrt{2pq}$. Then
		\begin{enumerate}[\rm a.]
			\item $p(x+1)$ is a square in $\NN$,
			\item  $\sqrt{2 \varepsilon_{2pq}}=y_1\sqrt{p}+y_2\sqrt{2q}$ and $2=py_1^2-2qy_2^2$,  for some integers $y_1$ and $y_2$.
		\end{enumerate}
		\item Let  $a$ and $b$   be two integers such that
		$ \varepsilon_{pq}=a+b\sqrt{pq}$. Then
		\begin{enumerate}[\rm a.]
			\item $2p(a+1)$ is a square in $\NN$,
			\item 	  $\sqrt{\varepsilon_{pq}}=b_1\sqrt{p} +b_2\sqrt{q}$ and 	$1= pb_1^2-qb_2^2$,  for some integers $b_1$ and $b_2$.
		\end{enumerate}
	\end{enumerate}
\end{lemma}	
\begin{proof}
	\begin{enumerate}[\rm 1.]
		\item It is known that $N(\varepsilon_{2pq}) = 1$. Then, by the unique factorization in $\ZZ$ and
		Lemma \ref{lm}, there exist some integers $y_{1}$ and $y_{2}$ $(y= y_{1}y_{1})$ such that
		$
		(1)\left \{
		\begin{array}{rcl}\label{1}
			x\pm 1&=&y_{1}^2, \\
			x\mp 1&=&2pqy_{2}^2,
		\end{array}
		\right.
		$
		$
		(2)\left \{
		\begin{array}{rcl}\label{2}
			x\pm 1&=&py_{1}^2, \\
			x\mp 1&=&2qy_{2}^2,
		\end{array}
		\right.
		$
		or 
		$
		(3)	\left \{
		\begin{array}{rcl}\label{3}
			x\pm 1&=&2py_{1}^2, \\
			x\mp 1&=&qy_{2}^2.
		\end{array}
		\right.
		$
		\begin{enumerate}[\rm $\star$]
			\item System (1) can not occur since it implies $1=(\frac{y_{1}^2}{p})=(\frac{x\pm 1}{p})=(\frac{x\mp 1\pm 2}{p})=(\frac{\pm 2}{p})=(\frac{2}{p})=-1,$  which is absurd.
			\item Similarly, system  (3) can not occur either since it implies $1=(\frac{q}{p})=(\frac{qy_{2}^2}{p})=(\frac{x\mp 1}{p})=(\frac{\pm 2}{p})=(\frac{2}{p})=-1,$  which is absurd.
			\item Suppose that  $\left \{
			\begin{array}{rcl}\label{2}
				x - 1&=&py_{1}^2, \\
				x+1&=&2qy_{2}^2.
			\end{array}
			\right.$ Then $1=(\frac{py_{1}^2}{q})=(\frac{x-1}{q})=(\frac{x+1-2}{q})=(\frac{-2}{q})=-1,$ which is also impossible.

			Thus, the only possible case is  $\left \{
			\begin{array}{rcl}\label{2}
				x +1&=&py_{1}^2, \\
				x- 1&=&2qy_{2}^2,
			\end{array}
			\right.$ which implies that $\sqrt{2\varepsilon_{2pq}}=y_1\sqrt{p} +y_2\sqrt{2q}$ and 	$2= py_1^2-2qy_2^2$.
		\end{enumerate}
		\item It is known that $N(\varepsilon_{pq}) = 1$. Then, by Lemma \ref{lm}, we have: 
		
		$
		(1)\left \{
		\begin{array}{rcl}\label{1}
			a\pm 1&=&pb_{1}^2, \\
			a\mp 1&=&qb_{2}^2,
		\end{array}
		\right.
		$
		$
		(2)\left \{
		\begin{array}{rcl}\label{2}
			a\pm 1&=&b_{1}^2, \\
			a\mp 1&=&pqb_{2}^2,
		\end{array}
		\right.
		$
		or
		$
		(3)	\left \{
		\begin{array}{rcl}\label{3}
			a\pm 1&=&2pb_{1}^2, \\
			a\mp 1&=&2qb_{2}^2,
		\end{array}
		\right.
		$
		
		for some integers $b_{1}$ and $b_{2}$  such that $(b= b_{1}b_{1})$ or $(b= 2b_{1}b_{1}).$
		
		\begin{enumerate}[\rm $\star$]
			\item System (1) can not occur since it implies $1=(\frac{q}{p})=(\frac{qb_{2}^2}{p})=(\frac{a\mp 1}{p})=(\frac{a\pm 1\mp 2}{p})=(\frac{\pm 2}{p})=(\frac{2}{p})=-1,$  which is absurd.
			\item Similarly, system  (2) can not occur either since it implies $1=(\frac{b_{1}^2}{p})=(\frac{a\pm 1}{p})=(\frac{\mp 2}{p})=(\frac{2}{p})=-1,$  which is absurd.
			\item Suppose that  $\left \{
			\begin{array}{rcl}\label{2}
				a - 1&=&2pb_{1}^2, \\
				a+ 1&=&2qb_{2}^2.
			\end{array}
			\right.$ Then $1=(\frac{2pb_{1}^2}{q})=(\frac{a-1}{q})=(\frac{a+1-2}{q})=(\frac{-2}{q})=-1,$ which is also impossible.

			Thus, the only possible case is  $\left \{
			\begin{array}{rcl}\label{2}
				a + 1&=&2pb_{1}^2, \\
				a-1&=&2qb_{2}^2,
			\end{array}
			\right.$ which implies that $\sqrt{\varepsilon_{pq}}=b_1\sqrt{p} +b_2\sqrt{q}$ and 	$1= pb_1^2-qb_2^2$.
		\end{enumerate}
	\end{enumerate}
\end{proof}

\begin{lemma}\label{lemmarank}
	 Let $p$, $q$ and $s$ be three different prime integers such that  $p\equiv -s\equiv 5\pmod 8$ and $q\equiv7\pmod 8$ with $\left(\frac{	p}{	q}\right)=\left(\frac{	p}{	s}\right)$. Then, The rank of the $2$-class group of $\LL^+$ is greater than or equal to $2$.
\end{lemma}
\begin{proof} 
	As we have  $E_{L_3}=\left\langle \varepsilon_{2},  \varepsilon_{ qs},
	\sqrt{\varepsilon_{ qs}\varepsilon_{2qs}}\right\rangle,$ then it is easy to check that $h_2(L_3)=1$.
	Therefore, the rank of the $2$-class group of $\LL$ is  $r_2(\mathbf{C}l(\LL^+))=t-1-e_d$, where ${e_d}$ is defined by $(E_{L_3}:E_{L_3}\cap N_{\LL^+/L_3}(\LL^+))=2^{ e_{d}}$ and $t=4$ is the number of ramified primes in $\LL^+/L_3$. Thus, the rank of the $2$-class group of $\LL^+$ is  $r_2(\mathbf{C}l(\LL^+))=3-e_{d}$.
	
	Let $\mathfrak p_{k}$ be a prime ideal of $k$ above $p$, where $k$ is a subfield of $L_3$. Notice that $p$ decomposes in $\mathbb{Q}(\sqrt{qs})$ and there are exactly $2$ prime ideal
	of $L_3$ laying above $p$. Using the well known properties of the norm residue symbols, we have:
	\begin{eqnarray*}
		\left(\frac{\varepsilon_2,\,ps}{ \mathfrak p_{L_3}}\right)=\left(\frac{\varepsilon_2,\,p}{ \mathfrak p_{L_3}}\right)&=&\left(\frac{N_{{L_3}/\mathbb Q(\sqrt{qs})}(\varepsilon_2),\,p}{ \mathfrak p_{\mathbb Q(\sqrt{qs})}}\right)\\
		&=&\left(\frac{-1,\,p}{ \mathfrak p_{\mathbb Q(\sqrt{qs})}}\right) =\left(\frac{-1,\,p}{ p}\right)  =1. 
	\end{eqnarray*}

	We similarly have $\left(\frac{\varepsilon_{qs},\,ps}{ \mathfrak p_{L_3}}\right)=\left(\frac{-1,\,ps}{ \mathfrak p_{L_3}}\right)=1$.
	It follows that $e_{d}\geq 1$ and so  $r_2(\mathbf{C}l(\LL^+))\geq 2$.
\end{proof}

Now we are able to state the first important result of this section.

\begin{theorem}   
		Let $p$, $q$ and $s$ be three different prime integers  satisfying (\ref{case1}). Let  $\ell\geq1$ be a positive odd  square-free integer relatively prime to $q$, $p$ and $s$. Put  $\LL^+=\mathbb{Q}(\sqrt{2},  \sqrt{pq}, \sqrt{ps})$ and $\LL=\mathbb{Q}(\sqrt{2},  \sqrt{pq}, \sqrt{ps}, \sqrt{-\ell})$. Then we have: 
	\begin{enumerate}[\rm 1.]
		\item	The unit group of $\LL^+$ is $$E_{\LL^+}=\left\langle-1, \varepsilon_{2},\varepsilon_{pq},  
		\sqrt{\varepsilon_{pq}\varepsilon_{2pq}}, \sqrt{\varepsilon_{ps}\varepsilon_{2ps}},  \sqrt{\varepsilon_{pq}\varepsilon_{ps}}, \sqrt{\varepsilon_{pq}\varepsilon_{qs}}, 
		\sqrt[4]{ \varepsilon_{pq}\varepsilon_{2pq}\varepsilon_{qs}\varepsilon_{2qs}} \right\rangle.$$
		Furthermore, the $2$-class group  of $\LL^+$ is isomorphic to $\ZZ/2\ZZ\times \ZZ/2\ZZ$.

		\item The unit group of $\LL$ is $$E_{\LL}=\left\langle   \eta, \varepsilon_{2},\varepsilon_{pq},  
		\sqrt{\varepsilon_{pq}\varepsilon_{2pq}}, \sqrt{\varepsilon_{ps}\varepsilon_{2ps}},  \sqrt{\varepsilon_{pq}\varepsilon_{ps}}, \sqrt{\varepsilon_{pq}\varepsilon_{qs}}, 
		\sqrt[4]{\varepsilon_{pq}\varepsilon_{2pq}\varepsilon_{qs}\varepsilon_{2qs}} \right\rangle,$$ where $\eta= \zeta_8	$ or $-1$ according to whether $\ell=1$ or not.
		
	\end{enumerate}
\end{theorem}
  \begin{proof} 
  	\begin{enumerate}[\rm 1.]
  			\item Let us compute the unit group of $\LL^+=\mathbb{Q}(\sqrt{2},  \sqrt{pq}, \sqrt{ps})$.

  		By  Lemma \ref{lm 4}, we have:
  	$$E_{L_2}=\left\langle \varepsilon_{2} ,    \varepsilon_{ ps},
  	\sqrt{\varepsilon_{ps}\varepsilon_{2ps}}\right\rangle.$$ 
  	
  	By Lemma \ref{lm 5}, we have:
  	
  	$$E_{L_3}=\left\langle \varepsilon_{2} ,    \varepsilon_{ qs},
  	\sqrt{\varepsilon_{ qs}\varepsilon_{2qs}}\right\rangle.$$ 
  	
  	By  Lemma \ref{lm 6}, we have:
  	$$E_{L_1}=\left\langle \varepsilon_{2} ,    \varepsilon_{ pq},
  	\sqrt{\varepsilon_{pq}\varepsilon_{2pq}}\right\rangle.$$

  	Thus
  	$$E_{L_1}E_{L_2}E_{L_3}=\left\langle-1, \varepsilon_{2} ,    \varepsilon_{ pq},  \varepsilon_{ps},     \varepsilon_{qs},	 
  	\sqrt{\varepsilon_{pq}\varepsilon_{2pq}},
  	\sqrt{\varepsilon_{ps}\varepsilon_{2ps}}, \sqrt{\varepsilon_{qs}\varepsilon_{2qs}} \right\rangle\cdot$$
  	
  	To find a fundamental system of units of $\LL^+$, it suffices, as we have said in the beginning of the section, to find elements $\xi$ of $E_{L_1}E_{L_2}E_{L_3}$ which are squares in $\LL^+$.
  	Put $$\xi^2= \varepsilon_2^a\cdot \varepsilon_{pq}^b \cdot \varepsilon_{ps}^c\cdot \varepsilon_{qs}^d\cdot \sqrt{\varepsilon_{pq}\varepsilon_{2pq}}^e \cdot  \sqrt{\varepsilon_{ps}\varepsilon_{2ps}}^f \cdot \sqrt{\varepsilon_{qs}\varepsilon_{2qs}}^g,
  	$$	
  	with $a,b,c,d,e,f,g\in\{0,1\}$.
  	We will use norm maps from $\LL^+$ to its biquadratic subextensions. The computations of these norms are summarized in the following table (see Table \ref{Tab 3}). Note that the third line of Table \ref{Tab 3}, is constructed as follows (we  similarly construct the rest of the table):

  	Furthermore, by Lemma \ref{lm 4}, we have:
  	
  	$$\left\{
  	\begin{array}{ll}
  		\sqrt{2 \varepsilon_{2ps}}=y_1\sqrt{p}+y_2\sqrt{2s} \quad \text{and} \quad 2=-py_1^2+2sy_2^2,\\
  		\sqrt{\varepsilon_{ps}}=b_1\sqrt{p} +b_2\sqrt{s} \quad \text{and} \quad	1= pb_1^2-sb_2^2.
  	\end{array}
  	\right.
  	$$
  	
  	Thus:
  	
  	$$\begin{array}{ll}
  		\sqrt{\varepsilon_{ps}\varepsilon_{2ps}}^{1+\sigma_1}&=	\sqrt{\varepsilon_{ps}\varepsilon_{2ps}}\sigma_1(\sqrt{\varepsilon_{ps}\varepsilon_{2ps}})\\
  		&=\varepsilon_{ps},\\
  		\vspace*{0.3cm}
  		\sqrt{\varepsilon_{ps}\varepsilon_{2ps}}^{1+\sigma_2}
  		&=\varepsilon_{ps}\varepsilon_{2ps},\\
  		\vspace*{0.3cm}
  		\sqrt{\varepsilon_{ps}\varepsilon_{2ps}}^{1+\sigma_3}
  		&=-1,\\
  		\vspace*{0.3cm}
  		\sqrt{\varepsilon_{ps}\varepsilon_{2ps}}^{1+\sigma_1\sigma_2}
  		&=\varepsilon_{ps},\\
  		\vspace*{0.3cm}
  		\sqrt{\varepsilon_{ps}\varepsilon_{2ps}}^{1+\sigma_1\sigma_3}
  		&=-\varepsilon_{2ps},\\
  		\vspace*{0.3cm}
  		\sqrt{\varepsilon_{ps}\varepsilon_{2ps}}^{1+\sigma_2\sigma_3}
  		&=-1.
  	\end{array}$$
  	
  	Using the same technique, we fill out the following table.

  	{\begin{table}[H]
  			\renewcommand{\arraystretch}{2.5}
  			
  			\begin{tabular}{|c|c|c|c|c|c|c|c|c|c}
  				\hline
  				$\varepsilon$ &$\varepsilon^{1+\sigma_1}$ & $\varepsilon^{1+\sigma_2}$ & $\varepsilon^{1+\sigma_3}$& $\varepsilon^{1+\sigma_1\sigma_2}$ &$\varepsilon^{1+\sigma_1\sigma_3}$ & $\varepsilon^{1+\sigma_2\sigma_3}$ \\ \hline
  				
  				$\varepsilon_{2}$ & $-1$ & $\varepsilon_{2}^2$  & $\varepsilon_{2}^2$ & $-1$ & $-1$ & $\varepsilon_{2}^2$  \\ \hline
  				
  				$\sqrt{\varepsilon_{pq}\varepsilon_{2pq}}$ & $-\varepsilon_{pq}$ & $1$ &$\varepsilon_{pq}\varepsilon_{2pq}$&$-\varepsilon_{2pq}$&$-\varepsilon_{pq}$&1 \\ \hline
  				$\sqrt{\varepsilon_{ps}\varepsilon_{2ps}}$ & $\varepsilon_{ps}$ & $\varepsilon_{ps}\varepsilon_{2ps}$ & $-1$ & $\varepsilon_{ps}$&$-\varepsilon_{2ps}$ & $-1$ \\ \hline
  				$\sqrt{\varepsilon_{qs}\varepsilon_{2qs}}$ & $-\varepsilon_{qs}$ & $1$ & $1$ & $-\varepsilon_{2qs}$&$-\varepsilon_{2qs}$ & $\varepsilon_{qs}\varepsilon_{2qs}$ \\ \hline
  			\end{tabular}
  			\caption{Norms in  $\LL^+/\QQ(\sqrt 2)$} \label{Tab 3}
  	\end{table} }

  Now we shall eliminate some forms of $\xi^2$ such that $\xi$ can not be in $\LL^+$.

  	\noindent\ding{229} Let us start by applying the norm   $N_{\LL^+/L_1}=1+\sigma_3$, where  $L_1=\mathbb{Q}(\sqrt{2}, \sqrt{pq})$. We have:
  	\begin{eqnarray*}
  		N_{\LL^+/L_1}(\xi^2)&=&  \varepsilon_{2}^{2a} \cdot\varepsilon_{pq}^{2b} \cdot 1\cdot 1\cdot (\varepsilon_{pq}\varepsilon_{2pq})^e \cdot (-1)^f \cdot 1.
  	\end{eqnarray*}
  	So $f=0$. Then $\xi^2=  \varepsilon_{2} \cdot \varepsilon_{pq}^b \cdot \varepsilon_{ps}^c\cdot \varepsilon_{qs}^d\cdot \sqrt{\varepsilon_{pq}\varepsilon_{2pq}}^e \cdot \sqrt{\varepsilon_{qs}\varepsilon_{2qs}}^g.$

  	\noindent\ding{229} Let us apply the norm  $N_{\LL^+/L_4}=1+\sigma_1$, where $L_4=\mathbb{Q}(\sqrt{pq}, \sqrt{ps})$. We have:
  	\begin{eqnarray*}
  		N_{\LL^+/L_4}(\xi^2)&=&  (-1)^{a}\cdot \varepsilon_{pq}^{2b} \cdot \varepsilon_{ps}^{2c}\cdot \varepsilon_{qs}^{2q}\cdot(-\varepsilon_{pq}^e) \cdot (-\varepsilon_{qs}^g)\\
  		&=&  (-1)^{a+e+g}\cdot \varepsilon_{pq}^{2b} \cdot \varepsilon_{ps}^{2c}\cdot \varepsilon_{qs}^{2q}\cdot \varepsilon_{pq}^e \cdot \varepsilon_{qs}^g.
  	\end{eqnarray*}
  	So $a+e+g=0 \pmod 2$ and $e=g$. Then $a=0$ and $\xi^2=  \varepsilon_{pq}^b \cdot \varepsilon_{ps}^c\cdot \varepsilon_{qs}^d\cdot \sqrt{\varepsilon_{pq}\varepsilon_{2pq}}^e \cdot \sqrt{\varepsilon_{qs}\varepsilon_{2qs}}^e.$ 
  	
  	Applying the other norms, we deduce non new information.

  	Notice that by Lemmas \ref{lm 4}, \ref{lm 5} and \ref{lm 6}, $\sqrt{\varepsilon_{pq}\varepsilon_{ps} }, \sqrt{\varepsilon_{pq}\varepsilon_{qs} },\sqrt{ \varepsilon_{ps}\varepsilon_{qs}}\in \LL^+$ and $\sqrt{\varepsilon_{pq}\varepsilon_{ps} }  \sqrt{\varepsilon_{pq}\varepsilon_{qs} }=\sqrt{ \varepsilon_{ps}\varepsilon_{qs}}$.
  	
  	We may assume that 
  	$\xi^2=\sqrt{\varepsilon_{pq} \varepsilon_{2pq}}^e \cdot \sqrt{\varepsilon_{qs} \varepsilon_{2qs}}^e.$

  	If $\sqrt{\varepsilon_{pq} \varepsilon_{2pq}}\cdot\sqrt{\varepsilon_{qs}\varepsilon_{2qs}}$ is not a square in $\LL^+$, then $q(\LL^+)=2^5$. On the other hand,  under conditions (\ref{case2}), we have  $ h_2(2)=h_2(qs)=1$ and $h_2(pq)=h_2(2pq) =h_2(ps)=h_2(2ps)=h_2(2qs)=2$ (cf. Lemma \ref{class numbers of quadratic field}). Therefore, by
  	Wada’s class number formula (cf. Lemma \ref{wada's f.}) we have: 
  	\begin{eqnarray*}
  		h_2(\LL^+)&=& \frac{1}{2^9}\cdot q(\LL^+) \cdot h_2(2) \cdot h_2(2ps) \cdot h_2(2qs) \cdot h_2(2pq)\cdot h_2(pq)\cdot h_2(ps)\cdot h_2(qs)\\
  		&=& \frac{1}{2^9}\cdot q(\LL^+) \cdot 1 \cdot 2 \cdot 2 \cdot 2\cdot 2\cdot 2\cdot 1= \frac{1}{2^4}q(\LL^+)=2,
  	\end{eqnarray*}
  	which is a contradiction, in fact, by Lemma \ref{lemmarank} we have $r_2(\mathbf{C}l(\LL^+))\geq 2$. Therefore,   $\sqrt{\varepsilon_{2qs}\varepsilon_{2qs}}\cdot\sqrt{\varepsilon_{pq} \varepsilon_{2pq}}$ is  a square in $\LL^+$, and so $q(\LL^+)=2^6$.  So the result.
  	
  	\item 	
  	
  		Let us now compute the unit group of $\LL=\mathbb{Q}(\sqrt{2},  \sqrt{pq}, \sqrt{ps}, \sqrt{-\ell}).$ 	To find a fundamental system of units of $\LL$, we shall use Lemmas \ref{Lemme azizi}, \ref{Lemme azizi2}. We distinguish the two following sub-cases:
  	\begin{enumerate}[\rm i.]
  	 \item   Assume that $\ell=1$, then we can use Lemma \ref{Lemme azizi} to construct a fundamental system of
  		units of $\LL$ from that of $\LL^+$. By the first item, we have: $$E_{\LL^+}=\left\langle-1, \varepsilon_{2},\varepsilon_{pq},  
  		\sqrt{\varepsilon_{pq}\varepsilon_{2pq}}, \sqrt{\varepsilon_{ps}\varepsilon_{2ps}},  \sqrt{\varepsilon_{pq}\varepsilon_{ps}}, \sqrt{\varepsilon_{pq}\varepsilon_{qs}}, 
  		\sqrt[4]{ \varepsilon_{pq}\varepsilon_{2pq}\varepsilon_{qs}\varepsilon_{2qs}} \right\rangle.$$
  	 So, let us 
  		put: \begin{equation}\label{eq1}
  			\xi^2= (2+\sqrt{2}) \cdot\varepsilon_2^a\cdot \varepsilon_{pq}^b \cdot \sqrt{ \varepsilon_{pq}\varepsilon_{2pq}}^c \cdot  \sqrt{ \varepsilon_{ps}\varepsilon_{2ps}}^d\cdot \sqrt{\varepsilon_{pq}\varepsilon_{ps}}^e \cdot \sqrt{\varepsilon_{pq}\varepsilon_{qs}}^f\cdot\sqrt[4]{ \varepsilon_{pq}\varepsilon_{2pq}\varepsilon_{qs}\varepsilon_{2qs}}^g,
  		\end{equation}
  		
  		with $a,b,c,d,e,f,g\in\{0,1\}$. Assume that $\xi\in \LL^+$.
  		We will use norm maps from $\LL^+$ to its biquadratic subextensions. The computations of these norms are summarized in    Tables \ref{Tab 3} and \ref{Tab 4} (see below).\\ 
  		By Lemmas \ref{lm 4}, \ref{lm 5} and \ref{lm 6}, we have: 
  		$$\left\{
  		\begin{array}{ll}
  			\sqrt{\varepsilon_{ps}}=b_1\sqrt{p}+b_2\sqrt{s}  \quad \text{and}   \quad		1= pb_1^2-sb_2^2,\\
  			\sqrt{ \varepsilon_{qs}}=a_1\sqrt{q}+a_2\sqrt{s}  \quad \text{and}   \quad	1=-qa_1^2+sa_2^2,\\
  				\sqrt{\varepsilon_{pq}}=y_1\sqrt{p} +y_2\sqrt{q} \quad \text{and}  \quad 1= py_1^2-qy_2^2.
  		\end{array}
  		\right.
  		$$

  		Thus:
  		
  		$$\begin{array}{ll}
  			\sqrt{\varepsilon_{pq}\varepsilon_{ps}}^{1+\sigma_1}&=	\sqrt{\varepsilon_{pq}\varepsilon_{ps}}\sigma_1(	\sqrt{\varepsilon_{pq}\varepsilon_{ps}})\\
  			&= \varepsilon_{pq}\varepsilon_{ps}.
  		\end{array}$$

  		We similarly have:
  		
  		{\begin{table}[H]
  				\renewcommand{\arraystretch}{2.5}
  				
  				{\footnotesize  \begin{tabular}{|c|c|c|c|c|c|c|c|c|c}
  						\hline
  						$\varepsilon$&$\varepsilon^{1+\sigma_1}$ & $\varepsilon^{1+\sigma_2}$ & $\varepsilon^{1+\sigma_3}$& $\varepsilon^{1+\sigma_1\sigma_2}$& $\varepsilon^{1+\sigma_1\sigma_3}$& $\varepsilon^{1+\sigma_2\sigma_3}$\\ \hline
  						
  						$\sqrt{\varepsilon_{pq}\varepsilon_{ps}}$&$\varepsilon_{pq}\varepsilon_{ps}$ &$\varepsilon_{ps}$&$\varepsilon_{pq}$& $\varepsilon_{ps}$&$\varepsilon_{pq}$&$1$\\ \hline
  						
  						$\sqrt{\varepsilon_{pq}\varepsilon_{qs}}$&$\varepsilon_{pq}\varepsilon_{qs}$ &$1$&$-\varepsilon_{pq}$& $1$&$-\varepsilon_{pq}$&$-\varepsilon_{qs}$\\ \hline
  						$\sqrt[4]{ \varepsilon_{pq}\varepsilon_{2pq}\varepsilon_{qs}\varepsilon_{2qs}}$&$(-1)^r\sqrt{\varepsilon_{pq}\varepsilon_{qs}}$ &$(-1)^s$&$(-1)^t\sqrt{\varepsilon_{pq}\varepsilon_{2pq}}$&$(-1)^u\sqrt{\varepsilon_{2pq}\varepsilon_{2qs}}$ &$(-1)^v\sqrt{\varepsilon_{pq}\varepsilon_{2qs}}$ &$(-1)^w\sqrt{\varepsilon_{qs}\varepsilon_{2qs}}$ \\ \hline
  						
  				\end{tabular}}
  				\caption{Norms in  $\LL^+/\QQ(\sqrt 2)$} \label{Tab 4}
  		\end{table}}
  		As above we shall apply norm maps to eliminate some forms of $\xi$. Using the above tables (\ref{Tab 3} and \ref{Tab 4}) one can easily check that.
  		
  		\noindent\ding{229} Let us start by applying the norm $N_{\LL^+/L_3}=1+\sigma_2\sigma_3$, where $L_3=\mathbb{Q}(\sqrt{2}, \sqrt{qs})$.
  		We have:
  		\begin{eqnarray*}
  			N_{\LL^+/L_3}(\xi^2)&=&  (2+\sqrt{2})^2 \cdot \varepsilon_{2}^{2a}\cdot 1\cdot 1 \cdot (-1)^d\cdot 1  \cdot(-\varepsilon_{qs})^f \cdot(-1)^{wg}\sqrt{\varepsilon_{qs}\varepsilon_{2qs}}^g\\
  			&=&  (-1)^{d+f+wg} \cdot(2+\sqrt{2})^2 \cdot \varepsilon_{2}^{2a}\cdot\varepsilon_{qs}^f\cdot\sqrt{\varepsilon_{qs}\varepsilon_{2qs}}^g.
  		\end{eqnarray*}
  		So $d+f+wg=0\pmod 2$.
  		By Lemma \ref{lm 5}, $E_{L_3}=\left\langle \varepsilon_{2} ,    \varepsilon_{ qs}, \sqrt{\varepsilon_{qs}\varepsilon_{2qs}}\right\rangle$. Since $\varepsilon_{qs}$, $\sqrt{\varepsilon_{qs}\varepsilon_{2ps}}$ and  $\varepsilon_{qs}\cdot\sqrt{\varepsilon_{qs}\varepsilon_{2qs}}$ are not squares in $L_3$, we have $f=g=0$ and $d=0$.

  		Therefore,
  		$\xi^2= (2+\sqrt{2}) \cdot\varepsilon_2^a\cdot \varepsilon_{pq}^b \cdot \sqrt{ \varepsilon_{pq}\varepsilon_{2pq}}^c \cdot  \sqrt{\varepsilon_{pq}\varepsilon_{ps}}^e.$

  		\noindent\ding{229} By applying the norm $N_{\LL^+/L_2}=1+\sigma_2$, where $L_2=\mathbb{Q}(\sqrt{2}, \sqrt{ps})$. We get:
  		\begin{eqnarray*}
  			N_{\LL^+/L_2}(\xi^2)&=& (2+\sqrt{2})^2 \cdot \varepsilon_{2}^{2a}\cdot 1 \cdot 1  \cdot(\varepsilon_{ps})^e.
  		\end{eqnarray*}
  		So $e=0$.
  		Therefore,
  		$\xi^2= (2+\sqrt{2}) \cdot\varepsilon_2^a\cdot \varepsilon_{pq}^b \cdot \sqrt{ \varepsilon_{pq}\varepsilon_{2pq}}^c.$
  		
  		\noindent\ding{229} Consider $L_4=\mathbb{Q}(\sqrt{pq}, \sqrt{ps})$, we will apply the norm $N_{\LL^+/L_4}=1+\sigma_1$. We have:
  		\begin{eqnarray*}
  			N_{\LL^+/L_4}(\xi^2)&=& 2 \cdot (-1)^a\cdot\varepsilon_{pq}^{2b} \cdot(-\varepsilon_{pq}^c)\\
  			&=&  (-1)^{a+c}\cdot\varepsilon_{pq}^{2b}\cdot 2 \cdot\varepsilon_{pq}^c.
  		\end{eqnarray*}
  		Since $2  \varepsilon_{pq}$  and $2$  are not squares in $L_4$, then the above equality is impossible. Therefore there is no element $\xi$ in $\LL^+$ satisfying the equality \eqref{eq1}. So the result by Lemma \ref{Lemme azizi}.
  		\item Assume that $\ell\not=1$. Now, we shall use Lemma \ref{Lemme azizi2}.
  		put: \begin{equation*}
  			\xi^2= \ell \cdot\varepsilon_2^a\cdot \varepsilon_{pq}^b \cdot \sqrt{ \varepsilon_{pq}\varepsilon_{2pq}}^c \cdot  \sqrt{ \varepsilon_{ps}\varepsilon_{2ps}}^d\cdot \sqrt{\varepsilon_{pq}\varepsilon_{ps}}^e \cdot \sqrt{\varepsilon_{pq}\varepsilon_{qs}}^f\cdot\sqrt[4]{ \varepsilon_{pq}\varepsilon_{2pq}\varepsilon_{qs}\varepsilon_{2qs}}^g,
  		\end{equation*}
  		Let us eliminate some cases. 
  		
  			\noindent\ding{229} Let us start by applying the norm $N_{\LL^+/L_3}=1+\sigma_2\sigma_3$, where $L_3=\mathbb{Q}(\sqrt{2}, \sqrt{qs})$.
  		We have:
  		\begin{eqnarray*}
  			N_{\LL^+/L_3}(\xi^2)&=&  \ell^2 \cdot \varepsilon_{2}^{2a}\cdot 1\cdot 1 \cdot (-1)^d\cdot 1  \cdot(-\varepsilon_{qs})^f \cdot(-1)^{wg}\sqrt{\varepsilon_{qs}\varepsilon_{2qs}}^g\\
  			&=&  (-1)^{d+f+wg} \cdot\ell^2 \cdot \varepsilon_{2}^{2a}\cdot\varepsilon_{qs}^f\cdot\sqrt{\varepsilon_{qs}\varepsilon_{2qs}}^g.
  		\end{eqnarray*}
  		So $d+f+wg=0\pmod 2$.
  		By Lemma \ref{lm 5}, $E_{L_3}=\left\langle \varepsilon_{2} ,    \varepsilon_{ qs}, \sqrt{\varepsilon_{qs}\varepsilon_{2qs}}\right\rangle$. Since $\varepsilon_{qs}$, $\sqrt{\varepsilon_{qs}\varepsilon_{2ps}}$ and  $\varepsilon_{qs}\cdot\sqrt{\varepsilon_{qs}\varepsilon_{2qs}}$ are not squares in $L_3$, we have $f=g=0$ and $d=0$.

  		Therefore,
  		$\xi^2=\ell \cdot\varepsilon_2^a\cdot \varepsilon_{pq}^b \cdot \sqrt{ \varepsilon_{pq}\varepsilon_{2pq}}^c \cdot  \sqrt{\varepsilon_{pq}\varepsilon_{ps}}^e.$

  		\noindent\ding{229} Let us start by applying the norm     $N_{\LL^+/L_2}=1+\sigma_2,$ where $L_2=\mathbb{Q}(\sqrt{2}, \sqrt{ps})$.
  		We have:
  		\begin{eqnarray*}
  			N_{\LL^+/L_2}(\xi^2)&=& \ell^2 \cdot \varepsilon_{2}^{2a}\cdot 1 \cdot 1  \cdot(\varepsilon_{ps})^e.
  		\end{eqnarray*}
  		So $e=0$.
  		Therefore,
  		$\xi^2= \ell\cdot\varepsilon_2^a\cdot \varepsilon_{pq}^b \cdot \sqrt{ \varepsilon_{pq}\varepsilon_{2pq}}^c.$
  		
  		\noindent\ding{229}  Let us apply the norm $N_{\LL^+/L_4}=1+\sigma_1$, where $L_4=\mathbb{Q}(\sqrt{pq}, \sqrt{ps})$. We get:
  		\begin{eqnarray*}
  			N_{\LL^+/L_4}(\xi^2)&=& \ell^2\cdot (-1)^a\cdot\varepsilon_{pq}^{2b} \cdot(-\varepsilon_{pq}^c)\\
  			&=& \ell^2\cdot (-1)^{a+c}\cdot\varepsilon_{pq}^{2b} \cdot\varepsilon_{pq}^c.
  		\end{eqnarray*}
  		So $a+c=0 \pmod 2$ and $c=0$, hence $a=c=0.$
  		Therefore,
  		$\xi^2= \ell\cdot \varepsilon_{pq}^b.$
  		
  		As $\sqrt{\varepsilon_{pq}}=b_1\sqrt{p} +b_2\sqrt{q}$ for some integer $b_1$ and $b_2$ (cf. Lemma \ref{lm 6})
  		$\ell \varepsilon_{pq}$ can not be a square in $\LL^+$. So the result by Lemma \ref{Lemme azizi2}.
  	\end{enumerate}	
  	\end{enumerate}
\end{proof} 
  
  \section{\textbf{The case: $p$, $q$ and $s$ satisfy conditions (\ref{case2})}}
  Let us start by collecting some results that will be useful in the sequel.
  
  \begin{lemma}\label{lm 1}
  	Let	$p\equiv 5\pmod 8$   and	$q\equiv  7\pmod 8$	 be two primes such that $\left(\frac{	p}{	q}\right)=-1$.
  	\begin{enumerate}[\rm 1.]
  		\item Let  $a$ and $b$   be two integers such that
  		$ \varepsilon_{pq}=a+b\sqrt{pq}$. Then
  		\begin{enumerate}[\rm a.]
  			\item $p(a-1)$ is a square in $\NN$,
  			\item 	  $\sqrt{2\varepsilon_{pq}}=b_1\sqrt{p} +b_2\sqrt{q}$ and 	$2= -pb_1^2+qb_2^2$,  for some integers $b_1$ and $b_2$.
  		\end{enumerate}

  		\item  There are two integers     $a$ and $b$   such that
  		$ \varepsilon_{2pq}=x+y\sqrt{2pq}$. Then
  		\begin{enumerate}[\rm a.]
  			\item $2p(x-1)$ is a square in $\NN$,
  			\item  $\sqrt{2 \varepsilon_{2pq}}=y_1\sqrt{2p}+y_2\sqrt{q}$ and $2=-2py_1^2+qy_2^2$,  for some integers $y_1$ and $y_2$.
  		\end{enumerate}
  	\end{enumerate}
  \end{lemma}	
  \begin{proof}
  	See \cite[Lemma 4.1]{chemszekhniniaziziUnits1}.
  \end{proof}
  \begin{lemma}\label{lm 2}
  	Let	$p\equiv -q\equiv5\pmod 8$ be two primes such that $\left(\frac{	p}{	q}\right)=-1$.
  	\begin{enumerate}[\rm 1.]
  		\item Let  $x$ and $y$   be two integers such that
  		$ \varepsilon_{2pq}=x+y\sqrt{2pq}$. Then
  		\begin{enumerate}[\rm a.]
  			\item $2p(x-1)$ is a square in $\NN$,
  			\item 	  $\sqrt{2\varepsilon_{2pq}}=y_1\sqrt{2p} +y_2\sqrt{q}$ and 	$2= -2py_1^2+qy_2^2$,  for some integers $y_1$ and $y_2$.
  		\end{enumerate}
  		\item  There are two integers     $a$ and $b$   such that
  		$ \varepsilon_{pq}=a+b\sqrt{pq}$. Then
  		\begin{enumerate}[\rm a.]
  			\item $p(a+1)$ is a square in $\NN$,
  			\item  $\sqrt{2\varepsilon_{pq}}=b_1\sqrt{p}+b_2\sqrt{q}$ and $2=pb_1^2-qb_2^2$,  for some integers $b_1$ and $b_2$.
  		\end{enumerate}
  	\end{enumerate}
  \end{lemma}	
  \begin{proof}
  	See \cite[Lemma 2.6]{ChemsUnits9}.
  \end{proof}

  \begin{lemma}\label{lm 3}
  	Let	$q_{1}\equiv 7\pmod 8$   and	$q_{2}\equiv  3\pmod 8$	 be two primes such that $\left(\frac{	q_{2}}{	q_{1}}\right)=-1$.
  	\begin{enumerate}[\rm 1.]
  		\item Let  $x$ and $y$   be two integers such that
  		$ \varepsilon_{q_{1}	q_{2}}=x+y\sqrt{q_{1}	q_{2}}$. Then
  		\begin{enumerate}[\rm a.]
  			\item $2q_{1}(x+1)$ is a square in $\NN$,
  			\item 	  $\sqrt{\varepsilon_{q_{1}	q_{2}}}=y_1\sqrt{q_{1}} +y_2\sqrt{q_{2}}$ and 	$1= q_{1}y_1^2-q_{2}y_2^2$,  for some integers $y_1$ and $y_2$ such that $y=2y_1y_2$.
  		\end{enumerate}

  		\item  There are two integers     $a$ and $b$   such that
  		$ \varepsilon_{2pq}=a+b\sqrt{2pq}$. Then
  		\begin{enumerate}[\rm a.]
  			\item $2q_{1}(a+1)$ is a square in $\NN$,
  			\item  $\sqrt{2 \varepsilon_{2q_{1}	q_{2}}}=b_1\sqrt{2q_{1}}+b_2\sqrt{q_{2}}$ and $2=2q_{1}b_1^2-q_{2}b_2^2$,  for some integers $b_1$ and $b_2$ such that $b=b_1b_2$.
  		\end{enumerate}
  	\end{enumerate}
  \end{lemma}	
  \begin{proof}
  	See \cite[Lemma 5]{chemszekhniniazizilambdas}.
  \end{proof}

  Now we are able to state the second important result of this section.
  \begin{theorem}  
  	Let  $p$, $q$ and $s$ be  three different prime integers  satisfying (\ref{case2}). Let  $\ell\geq1$ be a positive odd  square-free integer relatively prime to $q$, $p$ and $s$. Put $\LL^+=\mathbb{Q}(\sqrt{2},  \sqrt{pq}, \sqrt{ps})$ and $\LL=\mathbb{Q}(\sqrt{2},  \sqrt{pq}, \sqrt{ps}, \sqrt{-\ell}).$ Then we have: 
  	\begin{enumerate}[\rm 1.]
  		\item	$E_{\LL^+}=\left\langle -1, \varepsilon_{2}, \varepsilon_{pq}, \sqrt{\varepsilon_{pq}\varepsilon_{2pq}},  \sqrt{\varepsilon_{ps}\varepsilon_{2ps}}, \sqrt{\varepsilon_{qs}\varepsilon_{2qs}}, \sqrt{\varepsilon_{pq}\varepsilon_{ps}}, \sqrt{\varepsilon_{pq}\varepsilon_{qs}}\right\rangle\cdot$ 
  		\item
  		\begin{enumerate}[\rm i.] 
  			\item If $\ell=1.$  The group unit of $\LL$ is $\left\langle\zeta_8\right\rangle \times E_{\LL^+}$ or one of the following:
  			\begin{enumerate}[\rm $\bullet$]
  				
  				\item $\left\langle \zeta_8, \varepsilon_{2}, \varepsilon_{pq}, \sqrt{\varepsilon_{pq}\varepsilon_{2pq}}, \sqrt{\varepsilon_{ps}\varepsilon_{2ps}}, \sqrt{\varepsilon_{pq}\varepsilon_{ps}}, \sqrt{\varepsilon_{pq}\varepsilon_{qs}},  \sqrt{\zeta_8\varepsilon_{pq} \sqrt{ \varepsilon_{pq}\varepsilon_{2pq}}\sqrt{ \varepsilon_{qs}\varepsilon_{2qs}}}\right\rangle\cdot$
  				\item $\left\langle \zeta_8, \varepsilon_{2}, \varepsilon_{pq}, \sqrt{\varepsilon_{pq}\varepsilon_{2pq}}, \sqrt{\varepsilon_{ps}\varepsilon_{2ps}}, \sqrt{\varepsilon_{pq}\varepsilon_{ps}}, \sqrt{\varepsilon_{pq}\varepsilon_{qs}},  \sqrt{\zeta_8\sqrt{ \varepsilon_{pq}\varepsilon_{2pq}}\sqrt{ \varepsilon_{qs}\varepsilon_{2qs}}}\right\rangle,$
  				\item  $\left\langle \zeta_8, \varepsilon_{2}, \sqrt{\varepsilon_{pq}\varepsilon_{2pq}}, \sqrt{\varepsilon_{ps}\varepsilon_{2ps}}, \sqrt{\varepsilon_{pq}\varepsilon_{ps}}, \sqrt{\varepsilon_{pq}\varepsilon_{qs}},\sqrt{\zeta_8\sqrt{ \varepsilon_{pq}\varepsilon_{2pq}}\sqrt{ \varepsilon_{qs}\varepsilon_{2qs}}},  \sqrt{\zeta_8\varepsilon_{pq} \sqrt{ \varepsilon_{pq}\varepsilon_{2pq}}\sqrt{ \varepsilon_{qs}\varepsilon_{2qs}}}\right\rangle\cdot$
  				
  			\end{enumerate}
  			\item If $\ell\not=1.$
  			
  			$E_{\LL}=\left\langle  \eta, \varepsilon_{2}, \varepsilon_{pq}, \sqrt{\varepsilon_{pq}\varepsilon_{2pq}},  \sqrt{\varepsilon_{ps}\varepsilon_{2ps}}, \sqrt{\varepsilon_{qs}\varepsilon_{2qs}}, \sqrt{\varepsilon_{pq}\varepsilon_{ps}}, \sqrt{\varepsilon_{pq}\varepsilon_{qs}}\right\rangle\cdot$ 
  			
  		\end{enumerate}
  	\end{enumerate}
  \end{theorem}
  \begin{proof}

  	\begin{enumerate}[\rm 1.]
  		\item Let us compute the unit group of $\LL^+=\mathbb{Q}(\sqrt{2},  \sqrt{pq}, \sqrt{ps})$.
  		
  		By  Lemma \ref{lm 1}, we have:
  		$$E_{L_1}=\left\langle \varepsilon_{2} ,    \varepsilon_{ pq},
  		\sqrt{\varepsilon_{pq}\varepsilon_{2pq}}\right\rangle.$$
  		
  		By  Lemma \ref{lm 2}, we have:
  		$$E_{L_2}=\left\langle \varepsilon_{2} ,    \varepsilon_{ ps},
  		\sqrt{\varepsilon_{ps}\varepsilon_{2ps}}\right\rangle.$$
  		
  		By Lemma \ref{lm 3}, we have:
  		$$E_{L_3}=\left\langle \varepsilon_{2} ,    \varepsilon_{ qs},
  		\sqrt{\varepsilon_{ qs}\varepsilon_{2qs}}\right\rangle.$$

  		Thus
  		$$E_{L_1}E_{L_2}E_{L_3}=\left\langle-1, \varepsilon_{2} ,    \varepsilon_{ pq},  \varepsilon_{ps},     \varepsilon_{qs},	 
  		\sqrt{\varepsilon_{pq}\varepsilon_{2pq}},
  		\sqrt{\varepsilon_{ps}\varepsilon_{2ps}}, \sqrt{\varepsilon_{qs}\varepsilon_{2qs}} \right\rangle\cdot$$
  		
  		To find a fundamental system of units of $\LL^+$, it suffices, as we have said in the beginning of the section, to find elements $\xi$ of $E_{L_1}E_{L_2}E_{L_3}$ which are squares in $\LL^+$.
  		Put $$\xi^2= \varepsilon_2^a\cdot \varepsilon_{pq}^b \cdot \varepsilon_{ps}^c\cdot \varepsilon_{qs}^d\cdot \sqrt{\varepsilon_{pq}\varepsilon_{2pq}}^e \cdot  \sqrt{\varepsilon_{ps}\varepsilon_{2ps}}^f \cdot \sqrt{\varepsilon_{qs}\varepsilon_{2qs}}^g,
  		$$	
  		with $a,b,c,d,e,f,g\in\{0,1\}$.
  		We will use norm maps from $\LL^+$ to its biquadratic subextensions. The computations of these norms are summarized in the following table (see Table \ref{Tab 1}). Note that the third line of Table \ref{Tab 1}, is constructed as follows (we  similarly construct the rest of the table):\\

  		Furthermore, by Lemma \ref{lm 1}, we have:
  		
  		$$\left\{
  		\begin{array}{ll}
  			\sqrt{2 \varepsilon_{2pq}}=y_1\sqrt{2p}+y_2\sqrt{q} \quad \text{and} \quad 2=-2py_1^2+qy_2^2,\\
  			\sqrt{2\varepsilon_{pq}}=b_1\sqrt{p} +b_2\sqrt{q} \quad \text{and} \quad	2= -pb_1^2+qb_2^2.
  		\end{array}
  		\right.
  		$$
  		
  		Thus:
  		
  		$$\begin{array}{ll}
  			\sqrt{\varepsilon_{pq}\varepsilon_{2pq}}^{1+\sigma_1}&=	\sqrt{\varepsilon_{pq}\varepsilon_{2pq}}\sigma_1(\sqrt{\varepsilon_{pq}\varepsilon_{2pq}})\\
  			&=\varepsilon_{pq},\\
  			\vspace*{0.3cm}
  			\sqrt{\varepsilon_{pq}\varepsilon_{2pq}}^{1+\sigma_2}
  			&=1,\\
  			\vspace*{0.3cm}
  			\sqrt{\varepsilon_{pq}\varepsilon_{2pq}}^{1+\sigma_3}
  			&=\varepsilon_{pq}\varepsilon_{2pq},\\
  			\vspace*{0.3cm}
  			\sqrt{\varepsilon_{pq}\varepsilon_{2pq}}^{1+\sigma_1\sigma_2}
  			&=\varepsilon_{2pq},\\
  			\vspace*{0.3cm}
  			\sqrt{\varepsilon_{pq}\varepsilon_{2pq}}^{1+\sigma_1\sigma_3}
  			&=\varepsilon_{pq},\\
  			\vspace*{0.3cm}
  			\sqrt{\varepsilon_{pq}\varepsilon_{2pq}}^{1+\sigma_2\sigma_3}
  			&=1.
  		\end{array}$$
  		
  		Using the same technique, we fill out the following table.

  		{\begin{table}[H]
  				\renewcommand{\arraystretch}{2.5}
  				
  				\begin{tabular}{|c|c|c|c|c|c|c|c|c|c}
  					\hline
  					$\varepsilon$ &$\varepsilon^{1+\sigma_1}$ & $\varepsilon^{1+\sigma_2}$ & $\varepsilon^{1+\sigma_3}$& $\varepsilon^{1+\sigma_1\sigma_2}$ &$\varepsilon^{1+\sigma_1\sigma_3}$ & $\varepsilon^{1+\sigma_2\sigma_3}$ \\ \hline
  					
  					$\varepsilon_{2}$ & $-1$ & $\varepsilon_{2}^2$  & $\varepsilon_{2}^2$ & $-1$ & $-1$ & $\varepsilon_{2}^2$  \\ \hline
  					
  					$\sqrt{\varepsilon_{pq}\varepsilon_{2pq}}$ & $\varepsilon_{pq}$ & $1$ &$\varepsilon_{pq}\varepsilon_{2pq}$&$\varepsilon_{2pq}$&$\varepsilon_{pq}$&1 \\ \hline
  					$\sqrt{\varepsilon_{ps}\varepsilon_{2ps}}$ & $\varepsilon_{ps}$ & $\varepsilon_{ps}\varepsilon_{2ps}$ & $-1$ & $\varepsilon_{ps}$&$-\varepsilon_{2ps}$ & $-1$ \\ \hline
  					$\sqrt{\varepsilon_{qs}\varepsilon_{2qs}}$ & $\varepsilon_{qs}$ & $1$ & $1$ & $\varepsilon_{2qs}$&$\varepsilon_{2qs}$ & $\varepsilon_{qs}\varepsilon_{2qs}$ \\ \hline
  				\end{tabular}
  				\caption{Norms in  $\LL^+/\QQ(\sqrt 2)$} \label{Tab 1}
  		\end{table} }

  		Now we shall eliminate some forms of $\xi^2$ such that $\xi$ can not be in $\LL^+$.

  		\noindent\ding{229}  Let us start by applying the norm  $N_{\LL^+/L_1}=1+\sigma_3$, where $L_1=\mathbb{Q}(\sqrt{2}, \sqrt{pq})$. We have:
  		\begin{eqnarray*}
  			N_{\LL^+/L_1}(\xi^2)&=&  \varepsilon_{2}^{2a} \cdot\varepsilon_{pq}^{2b} \cdot 1\cdot 1\cdot (\varepsilon_{pq}\varepsilon_{2pq})^e \cdot (-1)^f \cdot 1.
  		\end{eqnarray*}
  		So $f=0$. Therefore, $\xi^2=  \varepsilon_{2} \cdot \varepsilon_{pq}^b \cdot \varepsilon_{ps}^c\cdot \varepsilon_{qs}^d\cdot \sqrt{\varepsilon_{pq}\varepsilon_{2pq}}^e \cdot \sqrt{\varepsilon_{qs}\varepsilon_{2qs}}^g.$

  		\noindent\ding{229} Consider    $L_4=\mathbb{Q}(\sqrt{pq}, \sqrt{ps})$,   we will  apply 	the norm     $N_{\LL^+/L_4}=1+\sigma_1$.
  		\begin{eqnarray*}
  			N_{\LL^+/L_4}(\xi^2)&=&  (-1)^{a}\cdot \varepsilon_{pq}^{2b} \cdot \varepsilon_{ps}^{2c}\cdot \varepsilon_{qs}^{2d}\cdot\varepsilon_{pq}^e \cdot \varepsilon_{qs}^g.
  		\end{eqnarray*}
  		So $a=0$. notice  $\sqrt{\varepsilon_{pq}}\notin L_4$, $\sqrt{\varepsilon_{qs}}\notin L_4$ and $\sqrt{\varepsilon_{pq}\varepsilon_{qs}}\notin L_4$, then $e=g=0$. Therefore, $\xi^2=  \varepsilon_{pq}^b \cdot \varepsilon_{ps}^c\cdot \varepsilon_{qs}^d$. 
  		Notice that by Lemmas \ref{lm 1}, \ref{lm 2} and \ref{lm 3}, $\sqrt{\varepsilon_{pq}\varepsilon_{ps} }, \sqrt{\varepsilon_{pq}\varepsilon_{qs} },\sqrt{ \varepsilon_{ps}\varepsilon_{qs}}\in \LL^+$ and $\sqrt{\varepsilon_{pq}\varepsilon_{ps} }  \sqrt{\varepsilon_{pq}\varepsilon_{qs} }=\sqrt{ \varepsilon_{ps}\varepsilon_{qs}}$.
  		So the result.

  		\item 	Let us now compute the unit group of $\LL=\mathbb{Q}(\sqrt{2},  \sqrt{pq}, \sqrt{ps}, \sqrt{-\ell}).$ 	To find a fundamental system of units of $\LL$, we shall use Lemmas \ref{Lemme azizi}, \ref{Lemme azizi2}. We distinguish the two following sub-cases:
  		\begin{enumerate}[\rm i.]
  			\item   Assume that $\ell=1$, then we can use Lemma \ref{Lemme azizi} to construct a fundamental system of
  			units of $\LL$ from that of $\LL^+$. By the first item, we have: 
  			
  			$E_{\LL^+}=\left\langle -1, \varepsilon_{2}, \varepsilon_{pq}, \sqrt{\varepsilon_{pq}\varepsilon_{2pq}},  \sqrt{\varepsilon_{ps}\varepsilon_{2ps}}, \sqrt{\varepsilon_{qs}\varepsilon_{2qs}}, \sqrt{\varepsilon_{pq}\varepsilon_{ps}}, \sqrt{\varepsilon_{pq}\varepsilon_{qs}}\right\rangle\cdot$ 
  			So, let us 
  			put: \begin{equation*}
  				\xi^2= (2+\sqrt{2}) \cdot\varepsilon_2^a\cdot \varepsilon_{pq}^b \cdot \sqrt{ \varepsilon_{pq}\varepsilon_{2pq}}^c \cdot  \sqrt{ \varepsilon_{ps}\varepsilon_{2ps}}^d\cdot  \sqrt{ \varepsilon_{qs}\varepsilon_{2qs}}^e\cdot\sqrt{\varepsilon_{pq}\varepsilon_{ps}}^f \cdot \sqrt{\varepsilon_{pq}\varepsilon_{qs}}^g,
  			\end{equation*}
  			
  			with $a,b,c,d,e,f,g\in\{0,1\}$. Assume that $\xi\in \LL^+$.
  			We will use norm maps from $\LL^+$ to its biquadratic subextensions. The computations of these norms are summarized in    Tables \ref{Tab 1} and \ref{Tab 2} (see below).\\ 
  			By Lemmas \ref{lm 1}, \ref{lm 2} and \ref{lm 3}, we have: 
  			$$\left\{
  			\begin{array}{ll}
  				\sqrt{2\varepsilon_{pq}}=y_1\sqrt{p} +y_2\sqrt{q} \quad \text{and}  \quad 2= -py_1^2+qy_2^2,\\
  				\sqrt{2 \varepsilon_{ps}}=b_1\sqrt{p}+b_2\sqrt{s}  \quad \text{and}   \quad		2= pb_1^2-sb_2^2,\\
  				\sqrt{ \varepsilon_{qs}}=a_1\sqrt{q}+a_2\sqrt{s}  \quad \text{and}   \quad	1=qa_1^2-sa_2^2.
  			\end{array}
  			\right.
  			$$

  			Thus:
  			
  			$$\begin{array}{ll}
  				\sqrt{\varepsilon_{pq}\varepsilon_{ps}}^{1+\sigma_1}&=	\sqrt{\varepsilon_{pq}\varepsilon_{ps}}\sigma_1(	\sqrt{\varepsilon_{pq}\varepsilon_{ps}})\\
  				&= \varepsilon_{pq}\varepsilon_{ps}.
  			\end{array}$$

  			We similarly have:
  			
  			{\begin{table}[H]
  					\renewcommand{\arraystretch}{2.5}
  					
  					{\footnotesize  \begin{tabular}{|c|c|c|c|c|c|c|c|c|c}
  							\hline
  							$\varepsilon$&$\varepsilon^{1+\sigma_1}$ & $\varepsilon^{1+\sigma_2}$ & $\varepsilon^{1+\sigma_3}$& $\varepsilon^{1+\sigma_1\sigma_2}$& $\varepsilon^{1+\sigma_1\sigma_3}$& $\varepsilon^{1+\sigma_2\sigma_3}$\\ \hline
  							
  							$\sqrt{\varepsilon_{pq}\varepsilon_{ps}}$&$\varepsilon_{pq}\varepsilon_{ps}$ &$-\varepsilon_{ps}$&$\varepsilon_{pq}$&$-\varepsilon_{ps}$&$\varepsilon_{pq}$& $-1$\\ \hline
  							
  							$\sqrt{\varepsilon_{pq}\varepsilon_{qs}}$&$-\varepsilon_{pq}\varepsilon_{qs}$ &$1$&$\varepsilon_{pq}$& $-1$&$-\varepsilon_{pq}$&$\varepsilon_{qs}$\\ \hline

  					\end{tabular}}
  					\caption{Norms in  $\LL^+/\QQ(\sqrt 2)$} \label{Tab 2}
  			\end{table}}
  			As above we shall apply norm maps to eliminate some forms of $\xi$. Using the above tables (\ref{Tab 1} and \ref{Tab 2}).
  			
  			\noindent\ding{229}  Let us start by applying the norm $N_{\LL^+/L_1}=1+\sigma_3$, where  $L_1=\mathbb{Q}(\sqrt{2}, \sqrt{pq}).$
  			We have:
  			\begin{eqnarray*}
  				N_{\LL^+/L_1}(\xi^2)&=&  (2+\sqrt{2})^2 \cdot \varepsilon_{2}^{2a}\cdot \varepsilon_{pq}^{2b} \cdot (\varepsilon_{pq}\varepsilon_{2pq})^c  \cdot(-1)^d\cdot1 \cdot\varepsilon_{pq}^f\cdot\varepsilon_{pq}^g.\\
  			\end{eqnarray*}
  			So $d=0$ and $f=g$.
  			Therefore,
  			$\xi^2= (2+\sqrt{2}) \cdot\varepsilon_2^a\cdot \varepsilon_{pq}^b \cdot \sqrt{ \varepsilon_{pq}\varepsilon_{2pq}}^c \cdot \sqrt{ \varepsilon_{qs}\varepsilon_{2qs}}^e\cdot\sqrt{\varepsilon_{pq}\varepsilon_{ps}}^f \cdot \sqrt{\varepsilon_{pq}\varepsilon_{qs}}^f.$

  			\noindent\ding{229}  Let us apply the norm $N_{\LL^+/L_2}=1+\sigma_2$, where $L_2=\mathbb{Q}(\sqrt{2}, \sqrt{ps})$. We have:
  			\begin{eqnarray*}
  				N_{\LL^+/L_2}(\xi^2)&=& (2+\sqrt{2})^2 \cdot \varepsilon_{2}^{2a}\cdot 1 \cdot 1  \cdot1\cdot(-\varepsilon_{ps})^f\cdot1.
  			\end{eqnarray*}
  			So $f=g=0$.
  			Therefore,
  			$\xi^2= (2+\sqrt{2}) \cdot\varepsilon_2^a\cdot \varepsilon_{pq}^b \cdot \sqrt{ \varepsilon_{pq}\varepsilon_{2pq}}^c \cdot \sqrt{ \varepsilon_{qs}\varepsilon_{2qs}}^e.$
  			
  			\noindent\ding{229} Consider $L_4=\mathbb{Q}(\sqrt{pq}, \sqrt{ps})$, we will apply the norm $N_{\LL^+/L_4}=1+\sigma_1$. We have:
  			\begin{eqnarray*}
  				N_{\LL^+/L_4}(\xi^2)&=& 2 \cdot (-1)^a\cdot\varepsilon_{pq}^{2b} \cdot\varepsilon_{pq}^c\cdot\varepsilon_{qs}^e.
  			\end{eqnarray*}
  			So $a=0$ and $c=e=1$.
  			Therefore,
  			$\xi^2= (2+\sqrt{2})\cdot \varepsilon_{pq}^b \cdot \sqrt{ \varepsilon_{pq}\varepsilon_{2pq}} \cdot \sqrt{ \varepsilon_{qs}\varepsilon_{2qs}}.$ Applying the other norms, we deduce non new information.
  			So the result by Lemma \ref{Lemme azizi}.
  			\item Assume that $\ell\not=1$. Now, we shall use Lemma \ref{Lemme azizi2}.
  			put: \begin{equation*}
  				\xi^2= \ell \cdot\varepsilon_2^a\cdot \varepsilon_{pq}^b \cdot \sqrt{ \varepsilon_{pq}\varepsilon_{2pq}}^c \cdot  \sqrt{ \varepsilon_{ps}\varepsilon_{2ps}}^d\cdot  \sqrt{ \varepsilon_{qs}\varepsilon_{2qs}}^e\cdot\sqrt{\varepsilon_{pq}\varepsilon_{ps}}^f \cdot \sqrt{\varepsilon_{pq}\varepsilon_{qs}}^g,
  			\end{equation*}
  			Let us eliminate some cases. 
  			
  			\noindent\ding{229} Let us start by applying the norm     $N_{\LL^+/L_1}=1+\sigma_3,$ where $L_1=\mathbb{Q}(\sqrt{2}, \sqrt{pq})$.
  			We have:
  			\begin{eqnarray*}
  				N_{\LL^+/L_1}(\xi^2)&=&  \ell^2 \cdot \varepsilon_{2}^{2a}\cdot \varepsilon_{pq}^{2b} \cdot (\varepsilon_{pq}\varepsilon_{2pq})^c  \cdot(-1)^d\cdot1 \cdot\varepsilon_{pq}^f\cdot\varepsilon_{pq}^g.\\
  			\end{eqnarray*}
  			So $d=0$ and $f=g$.
  			Therefore,
  			$\xi^2= \ell \cdot\varepsilon_2^a\cdot \varepsilon_{pq}^b \cdot \sqrt{ \varepsilon_{pq}\varepsilon_{2pq}}^c \cdot \sqrt{ \varepsilon_{qs}\varepsilon_{2qs}}^e\cdot\sqrt{\varepsilon_{pq}\varepsilon_{ps}}^f \cdot \sqrt{\varepsilon_{pq}\varepsilon_{qs}}^f.$

  			\noindent\ding{229} Consider $L_2=\mathbb{Q}(\sqrt{2}, \sqrt{ps})$, we will apply the norm $N_{\LL^+/L_2}=1+\sigma_2$. We have:
  			\begin{eqnarray*}
  				N_{\LL^+/L_2}(\xi^2)&=& \ell^2 \cdot \varepsilon_{2}^{2a}\cdot 1 \cdot 1  \cdot1\cdot(-\varepsilon_{ps})^f\cdot1.
  			\end{eqnarray*}
  			So $f=g=0$.
  			Therefore,
  			$\xi^2= \ell \cdot\varepsilon_2^a\cdot \varepsilon_{pq}^b \cdot \sqrt{ \varepsilon_{pq}\varepsilon_{2pq}}^c \cdot \sqrt{ \varepsilon_{qs}\varepsilon_{2qs}}^e.$
  			
  			\noindent\ding{229} Let us apply the norm $N_{\LL^+/L_4}=1+\sigma_1$, where $L_4=\mathbb{Q}(\sqrt{pq}, \sqrt{ps})$. We have:
  			\begin{eqnarray*}
  				N_{\LL^+/L_4}(\xi^2)&=& \ell^2 \cdot (-1)^a\cdot\varepsilon_{pq}^{2b} \cdot\varepsilon_{pq}^c\cdot\varepsilon_{qs}^e.
  			\end{eqnarray*}
  			So $a=0$. notice  $\sqrt{\varepsilon_{pq}}\notin L_4$, $\sqrt{\varepsilon_{qs}}\notin L_4$ and $\sqrt{\varepsilon_{pq}\varepsilon_{qs}}\notin L_4$, then $e=g=0$.
  			Therefore,
  			$\xi^2= \ell\cdot \varepsilon_{pq}^b.$

  			As $\sqrt{2\varepsilon_{pq}}=b_1\sqrt{p} +b_2\sqrt{q}$ for some integer $b_1$ and $b_2$ (cf. Lemma \ref{lm 1})
  			$\ell \varepsilon_{pq}$ can not be a square in $\LL^+$. So the result by Lemma \ref{Lemme azizi2}.
  		\end{enumerate}	
  	\end{enumerate}	
  \end{proof}

 \end{document}